\documentclass[11pt]{amsart}
\usepackage{amsmath, amssymb}
\usepackage{amsfonts}
\usepackage{graphicx}
\usepackage{mathrsfs}
\usepackage[arrow,matrix,curve,cmtip,ps]{xy}
\usepackage{amsthm}
\usepackage{enumerate}
\usepackage{color}
\usepackage[colorlinks]{hyperref}
\usepackage{tikz}

\allowdisplaybreaks

\newtheorem{thm}{Theorem}[section]
\newtheorem{lem}[thm]{Lemma}
\newtheorem{prop}[thm]{Proposition}
\newtheorem{cor}[thm]{Corollary}
\newtheorem*{theorem*}{Theorem}

\theoremstyle{remark}
\newtheorem{Remark}[thm]{Remark}
\newtheorem{note}[thm]{Notation}
\newtheorem{Definition}[thm]{Definition}
\newtheorem{example}[thm]{Example}

\numberwithin{equation}{section}

\newcommand{\Z}{\mathbb{Z}}
\newcommand{\N}{\mathbb{N}}

\newcommand{\T}{\mathbb{T}}
\newcommand{\la}{\Lambda}
\newcommand{\ul}{\overline{\Lambda}}

\newcommand{\per}{\widehat{\mathrm{Per}(\Lambda T)}}
\newcommand{\h}{\mathcal{H}_\Lambda}
\newcommand{\C}{\mathcal{C}}

\begin{document}
\title[Primitive ideals and decomposability of $k$-graph $C^*$-algebras]{Primitive ideal space of higher-rank graph $C^*$-algebras and decomposability}

\author[Hossein Larki]{Hossein Larki}

\address{Department of Mathematics\\
Shahid Chamran University of Ahvaz\\
Ahvaz, Iran}
\email{h.larki@scu.ac.ir}


\date{\today}



\begin{abstract}
In this paper, we describe primitive ideal space of the $C^*$-algebra $C^*(\Lambda)$ associated to any locally convex row-finite $k$-graph $\Lambda$. To do this, we will apply the Farthing's desourcifying method on a recent result of Carlsen, Kang, Shotwell, and Sims. We also characterize certain maximal ideals of $C^*(\la)$.

Furthermore, we study the decomposability of $C^*(\la)$. We apply the description of primitive ideals to show that if $I$ is a direct summand of $C^*(\la)$, then it is gauge-invariant and isomorphic to a certain $k$-graph $C^*$-algebra. So, we may characterize decomposable higher-rank $C^*$-algebras by giving necessary and sufficient conditions for the underlying $k$-graphs. Moreover, we determine all such $C^*$-algebras which can be decomposed into a direct sum of finitely many indecomposable $C^*$-algebras.
\end{abstract}

\maketitle


\section{Introductions}

Motivated from \cite{rob99}, the $C^*$-algebras of higher-rank graphs (or $k$-graphs) were introduced by Kumjian and Pask in \cite{kum00} as higher-rank analogous of the graph $C^*$-algebras. They were first considered in \cite{kum00} only for row-finite $k$-graphs with no sources, and then generalized for locally convex row-finite and finitely aligned setting \cite{rae03,rae04}. Since then, they have received a great deal of attention and provided a very interesting source of examples rather than ordinary graph $C^*$-algebras (see \cite{pas06,pas09} among others).

For any countable directed graph $E$, Hong and Szyma$\acute{\mathrm{n}}$ski described in \cite{hon04} primitive ideal space of the $C^*$-algebra $C^*(E)$ and its hull-kernel topology. After that, there have been many efforts to characterize primitive ideals of higher-rank graph $C^*$-algebras (see \cite{sho09,kan14} for example). The substantial work of Carlsen, Kang, Shotwell and A. Sims in \cite{car14} is to catalogue all primitive (two-sided and closed) ideals of the $C^*$-algebra $C^*(\la)$ of a row-finite higher-rank graph $\la$ with no sources. Despite some similarities, the structure of primitive ideals in higher-rank graph $C^*$-algebras are much more complicated compared with that of ordinary graphs.  Although the main result of \cite{car14} is a generalization of the Hong-Szyma$\acute{\mathrm{n}}$ski's description, but its methods and computations are quite different from \cite{hon04}.

In this paper, we let $\la$ be a locally convex row-finite $k$-graph with possible sources. Our primary aim is to characterize all primitive ideals of $C^*(\la)$. To this end, we apply the Farthing's desourcification \cite{far08} on the results of \cite{car14}. Recall that Farthing in \cite{far08} constructed a specific $k$-graph $\overline{\la}$ without sources which contains $\la$ as a subgraph. She showed in \cite[Theorem 3.30]{far08} that $C^*(\la)$ is a full corner in $C^*(\overline{\la})$, and therefore they are Morita-equivalent (see Section \ref{sec2.4} below for details). So, we may characterize the structure of primitive ideal space of $C^*(\la)$ using that of $C^*(\ul)$. Note that the Farthing's desourcification was modified by Webster in \cite[Section 4]{web11}. However, \cite[Proposition 4.12]{web11} shows that the desourcifications constructed in \cite[Section 3]{far08} and \cite[Section 4]{web11} are isomorphic when $\la$ is row-finite. Furthermore, there is a mistake in the proof of \cite[Theorem 3.30]{far08} (the proof works only for locally convex row-finite $k$-graphs), and Webster resolved it in \cite[Theorem 6.3]{web11} (see \cite[Remark 6.2]{web11}).

The rest of article will be devoted to some applications of the primitive ideal structure of $C^*(\la)$. First, as any ideal is an intersection of primitive ideals, certain maximal ideals of $C^*(\la)$ will be determined. Then, we will describe ideals which are direct summands of $C^*(\la)$, and study the decomposability of $C^*(\la)$.

This paper is organized as follows. In Section 2, we recall some elementary definitions and basic facts about $k$-graphs and their $C^*$-algebras from \cite{kum00,rae03}. We also review the Farthing's desourcification of a higher-rank graph $\la$. In Section 3, we define an equivalent relation on a row-finite $k$-graph $\la$, which will be used in Section 4 to describe generators of primitive ideals. We then discuss on relationships between equivalent paths in $\la$ and in its desourcification $\ul$. In Section 4, for any locally convex row-finite $k$-graph $\la$, we characterize primitive (two-sided and closed) ideals of $C^*(\la)$ and define specific irreducible representations whose kernels generate such ideals. Then, in Section 5, some certain maximal primitive ideals of $C^*(\la)$ are determined. As a consequence, we see that when $\la$ is a cofinal $k$-graph, then primitive ideals of $C^*(\la)$ are all maximal.

In Section 6, by applying the description of primitive ideals, we give some graph theoretic conditions for the decomposability of $C^*(\la)$. In particular, we show that if $C^*(\la)$ decomposes as $A\oplus B$, then $A$ and $B$ are gauge-invariant ideals of $C^*(\la)$, which are themselves isomorphic to certain $k$-graph $C^*$-algebras. Finally, in Section 7, we consider the question ``when is $C^*(\la)$ a direct sum of finitely many indecomposable $C^*$-algebras?". We describe all such higher-rank graph $C^*$-algebras by giving necessary and sufficient conditions on the underlying $k$-graphs.


\section{Higher-rank graphs and their $C^*$-algebras}

In this section, we review some basic facts about higher-rank graphs and their $C^*$-algebras, which will be needed in the paper. We refer the reader to \cite{kum00} and \cite{rae03} for more details.

\subsection{Higher-rank graphs}

Fix a positive integer $k>0$. We regard $\mathbb{N}^k:=\{n=(n_1,\ldots, n_k): n_i\geq 0\}$ as an additive semigroup with identity $0$. We denote by $e_1,\ldots,e_k$ the standard generators of $\N^k$. The relation $m\leq n ~ \Longleftrightarrow ~ m_i\leq n_i$ for $1\leq i\leq k$, puts a partial order on $\mathbb{N}^k$. We write $m\vee n$ and $m\wedge n$ for the coordinatewise maximum and minimum of $m,n\in \N^k$, respectively.

\begin{Definition}[\cite{kum00}]
A \emph{$k$-graph} (or \emph{graph of rank $k$}) is a countable small category $\la=(\la^0,\la^1,r,s)$ equipped with a \emph{degree functor} $d:\la \rightarrow \N^k$ satisfying the \emph{factorization property}: for each $\lambda\in\la$ and $m,n\in\N^k$ with $d(\lambda)=m+n$, there exist unique $\mu,\nu\in \la$ such that $d(\mu)=m$, $d(\nu)=n$, and $\lambda=\mu\nu$. Note that for $\mu,\nu\in\la$, the composition $\mu \nu$ makes sense only if $s(\mu)=r(\nu)$.
\end{Definition}

Note that every directed graph may be considered as a $1$-graph (and vise versa), in the usual manner. With this example in mind, we make some notations. For each $n\in \N^k$, we think of $\la^n:=d^{-1}(n)$ as the paths of degree $n$; in particular, $\la^0=d^{-1}(0)$ is the vertices in $\la$. For $v\in \la^0$, $H\subseteq \la^0$, and $E,F\subseteq \la$, we write $vE:=\{\mu\in E:r(\mu)=v\}$, $HE:=\{\mu\in E,r(\mu)\in H\}$ and
$$EF:=\{\mu\nu:\mu\in E,\nu\in F, ~\mathrm{and}~ s(\mu)=r(\nu)\},$$
and we define $Ev$ and $EH$ analogously. A vertex $v\in \la^0$ is called \emph{source} if $v\la^{e_i}=\emptyset$ for some $1\leq i\leq k$.

\begin{Definition}
Let $\la$ be a $k$-graph. We say that $\la$ is \emph{row-finite} if the sets $v\la^n$ are all finite for $v\in\la^0$, $n\in\N^k$. Also, $\la$ is called \emph{locally convex}, if for every $v\in \Lambda^{0}$, $i\neq j\in \{1,2,\ldots,k\}$, $\lambda\in v\Lambda^{e_{i}}$ and $\mu\in v\Lambda^{e_{j}}$, we have $s(\lambda)\Lambda^{e_{j}}\neq \emptyset$ and $s(\mu)\Lambda^{e_{i}}\neq \emptyset$. Observe that if $\la$ has no sources, then $\la$ is locally convex. Throughout the paper, we work only with locally convex, row-finite $k$-graphs.
\end{Definition}

\subsection{Boundary Paths}

Let $\la$ be a locally convex row-finite $k$-graph. For $n\in \N^k$, we write
$$\la^{\leq n}:=\left\{\lambda\in\la: d(\lambda)\leq n, ~ \mathrm{and} ~ d(\lambda)_i<n_i ~\Longrightarrow ~ s(\lambda)\la^{e_i}=\emptyset \right\}.$$
Note that if $\la$ has no sources, then $\la^{\leq n}=\la^n$.

\begin{example}
For any $m\in (\N\cup \{\infty\})^k$, let $\Omega_{k,m}:=\{(p,q)\in \N\times \N: p\leq q \leq m\}$. If we define $r(p,q):=(p,p)$, $s(p,q):=(q,q)$, and $d(p,q):=q-p$, then $\Omega_{k,m}$ is a $k$-graph. For simplicity, each $(p,p)$ is denoted by $p$. Thus, we regard $\Omega_{k,m}^0=\{p: p\leq m\}$ as the object set of $\Omega_{k,m}$. For $m=(\infty,\ldots,\infty)$, the $k$-graph $\Omega_{k,m}$ is denoted by $\Omega_k$ in \cite{kum00}.
\end{example}

A \emph{boundary path in $\la$} is a degree preserving functor $x: \Omega_{k,m}\rightarrow\Lambda$ such that $p\leq m$ and $p_{i}=m_{i}$ imply $x(p)\Lambda^{e_{i}}=\emptyset$. Then, $m$ is called the \emph{degree of $x$} and we write $d(x)=m$. When $m=(\infty,\ldots,\infty)$, $x$ is an \emph{infinite path} in the sense of \cite{kum00}. We denote by $\la^{\leq\infty}$ the set of boundary paths in $\la$. Note that if $\la$ has no sources, then every boundary path is an infinite path, so $\la^{\leq \infty}=\la^\infty$.

Given any $x\in \la^{\leq \infty}$ and $n\leq d(x)$, we may define the $n$-shifted boundary path $\sigma^n(x):\Omega_{k,d(x)-n}\rightarrow \la$ by $\sigma^{n}(x)(p,q) :=x(p+n,q+n)$ for $p\leq q\leq d(x)-n$. Moreover, if $\lambda\in\Lambda x(0)$, there is a unique boundary path $\lambda x\in \la^{\leq \infty}$ such that $\lambda x (0,d(\lambda))=\lambda$ and $(\lambda x)(d(\lambda), d(\lambda) +p)=x(0,p)$ for all $p\leq d(x)$. So, we have $\sigma^{d(\lambda)}(\lambda x)=x$.

\subsection{The $C^*$-algebra of a higher-rank graph}

Let $\la$ be a locally convex row-finite $k$-graph. A \emph{Cuntz-Krieger $\la$-family} is a set of partial isometries $\{S_\lambda: \lambda\in\la\}$ satisfying the following relations:
\begin{enumerate}
  \item $S_v S_w=\delta_{v,w}S_v$ for all $v,w\in \la^0$;
  \item $S_\lambda S_{\lambda'}=S_{\lambda\lambda'}$ for all $\lambda,\lambda'\in\la$ with $s(\lambda)=r(\lambda')$;
  \item $S_\lambda^* S_\lambda=S_{s(\lambda)}$ for all $\lambda\in \la$;
  \item $S_v=\sum_{\lambda\in v\la^{\leq n}} S_\lambda S_\lambda^*$ for all $v\in \la^0$ and $n\in \N^k$.
\end{enumerate}
The associated $C^*$-algebra $C^*(\la)$ is the \emph{universal $C^*$-algebra} generated by a Cuntz-Krieger $\la$-family $\{s_\lambda:\lambda\in \la\}$. The universality implies that there is a \emph{gauge action} $\gamma:\mathbb{T}^k\rightarrow \mathrm{Aut} C^*(\la)$ such that $\gamma_t(s_\lambda)=t^{d(\lambda)} s_\lambda$ for $t\in \mathbb{T}^k$, where $t^{d(\lambda)}:=t_1^{d(\lambda)_1} \ldots t_k^{d(\lambda)_k}$.

By \emph{ideal} we mean a closed and two-sided one. An ideal $I$ of $C^*(\la)$ is called \emph{gauge-invariant} if $\gamma_z(I)\subseteq I$ for every $z\in \mathbb{T}^k$. It is well-know by \cite[Theorem 5.2]{rae03} that gauge-invariant ideals of $C^*(\la)$ are associated to hereditary and saturated subsets of $\la^0$.

\begin{Definition}
A subset $H$ of $\la^{0}$ is called \emph{hereditary} if $r(\lambda)\in H$ implies $s(\lambda)\in H$ for every $\lambda\in \la$. Also, we say that $H$ is \emph{saturated} if $v\in \la^0$ and $s(v\la^{\leq n})\subseteq H$ for some $n\in\N^k$, then $v\in H$. If $H$ is a subset of $\la^0$, the saturated closure of $H$ is the smallest saturated subset $\Sigma(H)$ of $\la^0$ containing $H$. Recall from \cite[Lemma 5.1]{rae03} that if $H\subseteq \la^0$ is hereditary, then so is $\Sigma(H)$. We denote by $\mathcal{H}_\la$ the collection of saturated hereditary subsets of $\la^0$. Note that $\mathcal{H}_\la$ has a lattice structure by the operations
\begin{align*}
H_1\wedge H_2&:=H_1\cap H_2\\
H_1\vee H_2&:=\Sigma\big(H_1\cup H_2\big).
\end{align*}
\end{Definition}

For every $H\in \mathcal{H}_\la$, we denote by $I_H$ the gauge-invariant ideal of $C^*(\la)$ generated by $\{s_v:v\in H\}$, which is
$$I_{H}=\overline{\mathrm{span}}\{s_{\lambda}s_{\lambda'}^*: \lambda,\lambda'\in\la, ~ s(\lambda)=s(\lambda')\in H\}.$$
Also, \cite[Theorem 5.2(b)]{rae03} shows that $\la\setminus \la H$ is a locally convex $k$-subgraph of $\la$ and we have $C^*(\la\setminus \la H)\cong C^*(\la)/I_H$. According to \cite[Theorem 5.2(a)]{rae03} the map $H\mapsto I_H$ is a lattice isomorphism from $\mathcal{H}_\la$ onto the set of gauge-invariant ideals of $C^*(\la)$. Moreover, for each ideal $I$ of $C^*(\Lambda)$, the set $H_{I}:=\{v\in \Lambda^{0} : s_{v}\in I\}$ is a saturated hereditary subset of $\la^0$.

\subsection{Removing sources from a higher-rank graph}\label{sec2.4}

Here, we briefly review the desourcification constructions of Farthing \cite{far08} and Webster \cite{web11} with some minor modifications. We refer the reader to \cite{far08} and \cite{web11} for details and proofs. Note that, in case $\la$ is a row-finite $k$-graph, \cite[Proposition 4.12]{web11} shows that the constructions of \cite[Section 4]{web11} and \cite[Section 3]{far08} produce isomorphic desourcifications.

Let $\la$ be a locally convex row-finite $k$-graph. Let $P_{\Lambda}$ be the set
$$P_{\Lambda}:=\{ (x;(m,n)) : x\in \Lambda^{\leq \infty} ~ \mathrm{and}~ m\leq n\in \N^{k}\}.$$
We define $(x;(m,n))\approx (y;(p,q))$ if and only if
\begin{enumerate}
\item[P1)] $ x(m\wedge d(x), n\wedge d(x))= y(p\wedge d(y), q\wedge d(y))$,
\item[P2)] $ m-m\wedge d(x)=p-p\wedge d(y)$, and
\item[P3)] $ n-m=q-p.$
\end{enumerate}
Then, $\approx$ is an equivalence relation on $P_{\Lambda}$ and we denote the equivalence class of $(x;(m,n))$ by $[x;(m,n)]$. According to \cite[Theorem 3.24]{far08}, $\overline{\Lambda}:=P_\la/\approx$ equipped with
\begin{align*}
&r([x;(m,n)]):=[x;(m,m)],\\
&s([x;(m,n)]):=[x;(n,n)], ~ \mathrm{and}\\
&d([x;(m,n)]):=n-m.
\end{align*}
is a $k$-graph containing no sources. Also, the composition of paths in $\ul$ is of the form
$$[x;(m,n)] \circ [y;(p,q)]=[x(0,n\wedge d(x))\sigma^{p\wedge d(y)}; (m,n+q-p)],$$
which makes sense only if $[x;(n,n)]=[y;(p,p)]$.
For simplicity, we usually denote each $(x;(m,m))$ by $(x;m)$, so $\overline{\la}^0=\{[x;m]: x\in \la^{\leq \infty}, m\in \N^k\}$ is the vertex set of $\overline{\la}$. Note that, for $(x;m)\approx (y;n)$ it suffices to check only (P1) and (P2) because (P3) is trivial.

The correspondence $\lambda\mapsto [\lambda x;(0,d(\lambda))]$, where $x\in s(\lambda)\la^{\leq\infty}$, is an injective $k$-graph morphism from $\la$ into $\ul$ \cite[Poposition 4.13]{web11}. Hence, we may regard $\la$ as a subgraph of $\overline{\la}$. Also, the map $\pi :\ul\rightarrow \Lambda$ defined by $\pi([x;(m,n)])=[x;(m\wedge d(x),n\wedge d(x))]$, for $x\in\la^{\leq\infty}$ and $m\leq n\in\N^k$, is a well-defined surjective $k$-graph morphism such that $\pi \circ\pi=\pi$. Recall from \cite[Theorem 3.26]{far08} that if $\la$ is row-finite, then so is $\overline{\la}$. The $k$-graph $\overline{\la}$ is called the \emph{desourcification of $\Lambda$}.

For every $x\in \la^{\leq \infty}$ and $p\in \N^k$, we may define an infinite path $[x;(p,\infty)]:\Omega_k\rightarrow \ul$ by
$$[x;(p,\infty)](m,n):=[x;(p+m,p+n)]$$
for all $m\leq n\in \N^k$. The following Lemma says that every boundary path in $\ul$ (which is an infinite path) is of the form $[x;(p,\infty)]$.

\begin{lem}[See {\cite[Proposition 2.3]{Rob09}}]\label{lem2.4}
Let $\la$ be a locally convex row-finite $k$-graph and let $\ul$ be the desourcification of $\la$. For each $z\in \ul^\infty$, there exist $\pi(z)\in \la^{\leq \infty}$ and $p_z\in \N^k$ such that $p_z\wedge d(\pi(z))=0$ and $z=[\pi(z);(p_z,\infty)]$. Moreover, if $z(0)\in \la^0$, then $p_z=0$.
\end{lem}

\begin{proof}
The first statement is \cite[Proposition 2.3]{Rob09}. For the second, if $[\pi(z);p_z]=z(0)\in \la^0$, then $\pi([\pi(z);p_z])=[\pi(z);p_z]$. So, we have
$$[\pi(z);p_z]=\pi\big([\pi(z);p_z]\big)=[\pi(z);p_z\wedge d(\pi(z))]=[\pi(z);0].$$
Now condition (P2) for $(\pi(z);p_z)\approx (\pi(z);0)$ implies that $p_z=p_z-p_z\wedge d(\pi(z))=0$.
\end{proof}

In other words, we may define the map $\pi:\ul^\infty \rightarrow \la^{\leq \infty}$ by
$$\pi([x;(p,\infty)])(m,n):=x(p\wedge d(x)+m,p\wedge d(x)+n)$$
for $m\leq n\leq d(x)-p\wedge d(x)$. Thus, we have $\pi([x;(p,\infty)])=\sigma^{p\wedge d(x)}(x)$.

If $\{s_\lambda:\lambda\in \overline{\la}\}$ is a generating Cuntz-Krieger $\overline{\la}$-family for $C^*(\overline{\la})$, then $\{s_\lambda:\lambda\in \la\}$ is a universal Cuntz-Krieger $\la$-family \cite[Theorem 3.28]{far08}. Hence, $C^*(\{s_\lambda:\lambda\in \la\})=C^*(\la)$. Moreover, if we set $P:=\sum_{v\in\la^0}s_v$ as a projection in the multiplier algebra $\mathcal{M}(C^*(\overline{\la}))$, then \cite[Theorem 3.30]{far08} and \cite[Theorem 6.3]{web11} say that $P C^*(\overline{\la}) P$ is a full corner in $C^*(\ul)$ such that $PC^*(\ul) P=C^*(\la)$. In particular, the $C^*$-algebras $C^*(\la)$ and $C^*(\ul)$ are Morita-equivalent. Therefore, the map $J\mapsto PJP$ is an isomorphism from the lattice of ideals in $C^*(\ul)$ onto that of $C^*(\la)$.

\begin{Remark}
There is an error in the proof \cite[Theorem 3.30]{far08}; however, the proof may be considered for locally convex row-finite $k$-graphs. Webster resolved this result in \cite[Theorem 6.3]{web11} (see \cite[Remark 6.2]{web11}).
\end{Remark}


\section{equivalent paths and periodicity}

For any 1-graph $E$, Hong and Szyma$\acute{\mathrm{n}}$ski in \cite{hon04} described primitive ideals of $C^*(E)$ by specific collections of vertices $T$ in $E$, called \emph{maximal tails}, and periodicity in the quotient graphs $E T$. Recall that periodicity in  1-graphs can be determined by cycles with no entrances \cite[Lemma 2.9]{kan14}. Although we know that if a $k$-graph $\la$ is periodic (see Definition \ref{defn3.2} below), then it contains a generalized cycle with no entrances \cite[Lemma 4.4]{lar17}, but structure of periodic $k$-graphs is more complicated compared with 1-graphs and the arguments of \cite{hon04} could not be generalized for $k$-graphs in general.

To deal with the periodicity, Carlsen et al. in \cite{car14} used the following equivalent relation on $\la$ which is inspired from \cite{dav09}.

\begin{Definition}[\cite{car14}]
Let $\la$ be a row-finite $k$-graph. We set an equivalent relation on $\la$ by
$$\mu\sim_{\la} \nu ~~ \Longleftrightarrow ~~ s(\mu)=s(\nu) ~~ \mathrm{ and} ~~ \mu x=\nu x ~ \mathrm{for ~ all ~} x\in s(\mu)\la^{\leq\infty}.$$
(The subscript in $\sim_\la$ indicates the underlying $k$-graph.) We then define
$$\mathrm{Per}(\la):=\{d(\mu)-d(\nu): \mu,\nu\in\la ~ \mathrm{and}~ \mu\sim_{\la}\nu\},$$
which is a subset of $\mathbb{Z}^k$.
\end{Definition}

There are several conditions for aperiodicity in the literature, which are equivalent for finitely aligned $k$-graphs (cf. \cite[Proposition 2.11]{sho12}). We consider the following from \cite{rae03} which was called \emph{Condition (B)} there.

\begin{Definition}\label{defn3.2}
Let $\la$ be a locally convex row-finite $k$-graph. We say that $\la$ is \emph{aperiodic} if for every $v\in \la^0$, there exists $x\in v\la^{\leq \infty}$ such that $\mu\neq\nu \in \la v$ implies $\mu x\neq \nu x$. If $\la$ is not aperiodic, it is called \emph{periodic}.
\end{Definition}

\begin{prop}\label{prop3.3}
Let $\la$ be a locally convex row-finite $k$-graph. Then the following are equivalent.
\begin{enumerate}
  \item $\la$ is aperiodic.
  \item $\mathrm{Per}(\la)=\{0\}$.
  \item There are no distinct $\mu,\nu\in \la$ such that $\mu\sim_\la \nu$.
\end{enumerate}
\end{prop}

\begin{proof}
Implications (1) $\Leftrightarrow$ (3) and (3) $\Rightarrow$ (2) are immediate. To see (2) $\Rightarrow $ (3), let $\mu\sim_\la \nu$. Then $d(\mu)-d(\nu)\in \mathrm{Per}(\la)=\{0\}$, and so $d(\mu)=d(\nu)$. Fix $x\in s(\mu)\la^{\leq \infty}$. Then for $m:=d(\mu)=d(\nu)$ we have $\mu=\mu x(0,m)=\nu x(0,m)=\nu$, giving (3).
\end{proof}

In the next proposition, we will relate equivalent paths in a $k$-graph $\la$ with those in its desourcification $\overline{\la}$. Before that, we state a simple lemma.

\begin{lem}\label{lem3.4}
Let $\la$ be a row-finite $k$-graph and $\ul$ the desourcification of $\la$. Suppose $\mu,\nu\in\ul$ such that $\mu\sim_{\ul} \nu$. If $\mu=[x;(m,m+d(\mu))]$ for some $x\in \la^{\leq \infty}$ and $m\in \N^k$, then $\nu=[x;(m,m+d(\nu))]$.
\end{lem}

\begin{proof}
By $\mu\sim_{\ul} \nu$, we can write
\begin{align*}
[x;(m,\infty)]&=[x;(m,m+d(\mu))] \circ [x;(m+d(\mu),\infty)]\\
&= \mu \circ [x;(m+d(\mu),\infty)]\\
&= \nu \circ [x;(m+d(\mu),\infty)].
\end{align*}
So, we must have $\nu=[x;(m,m+d(\nu))]$ by the factorization property.
\end{proof}

\begin{prop}\label{prop3.5}
Let $\Lambda$ be a locally convex, row-finite $k$-graph and let $\ul$ be the desourcification of $\la$. For any $\mu,\nu\in\ul$ with $r(\mu)=r(\nu)$, the following are equivalent:
\begin{enumerate}
  \item $\mu\sim_{\ul} \nu$;
  \item $\pi(\mu)\sim_{\ul} \pi(\nu)$ and $d(\mu)-d(\pi(\mu))=d(\nu)-d(\pi(\nu))$;
  \item $\pi(\mu)\sim_{\la} \pi(\nu)$ and $d(\mu)-d(\pi(\mu))=d(\nu)-d(\pi(\nu))$.
\end{enumerate}
\end{prop}

\begin{proof}
Write $\mu=[x;(m,m+d(\mu))]$ and $\nu=[y;(n,n+d(\nu))]$. For $q:=m\wedge d(x)$, we have
$$[x;(m,m+d(\mu))]=[\sigma^q(x);(m-q,m-q+d(\mu))]$$
such that
$$(m-q)\wedge d(\sigma^q(x))=\big(m-m\wedge d(x)\big)\wedge \big(d(x)-m\wedge d(x)\big)=m\wedge d(x)- m\wedge d(x)=0.$$
So, without loss of generality, we may suppose $\mu=[x;(m,m+d(\mu))]$ and $\nu=[y;(n,n+d(\nu))]$ such that $m\wedge d(x)=n \wedge d(y)=0$. Then
$$\pi(\mu)=[x;(m\wedge d(x),(m+d(\mu))\wedge d(x))]=[x;(0,d(\mu)\wedge d(x))]$$
and similarly $\pi(\nu)=[y;(0,d(\nu)\wedge d(y))]$. Moreover, we have $[x;(0,m)]=[y;(0,m)]$. Indeed, conditions (P1) and (P2) for $[x;m]=r(\mu)=r(\nu)=[y;n]$ imply $x(0)=y(0)$ and
$$m=m- m\wedge d(x)=n-n\wedge d(y)=n,$$
giving $(x;(0,m))\approx (y;(0,m))$.

(1) $\Longrightarrow$ (2): We suppose $\mu \sim_{\ul} \nu$. By Lemma \ref{lem3.4}, $\nu=[x;(m,m+d(\nu))]$ and $\pi(\nu)=[x;(0,d(\nu)\wedge d(x))]$. So, (P2) for $[x;m+d(\mu)]=s(\mu)=s(\nu)=[x;m+d(\nu)]$ yields that
\begin{align*}
m+d(\mu)-\big(m+d(\mu)\big)\wedge d(x)&=m+d(\nu)-\big(m+d(\nu)\big)\wedge d(x)\\
\Longrightarrow \hspace{1.5cm} d(\mu)-d(\pi(\mu))&=d(\nu)-d(\pi(\nu)).
\end{align*}

Now we prove $\pi(\mu)\sim_{\ul} \pi(\nu)$. Using Lemma \ref{lem2.4}, every infinite path in $\ul$ is of the form $[z;(p_z;\infty )]$ with $p_z\wedge d(z)=0$. Fix $[z;(p_z,\infty)]\in s(\pi(\mu))\ul^\infty$. In view of Lemma \ref{lem2.4}, we can assume $p_z=0$ because $[z;p_z]=s(\pi(\mu))\in \la^0$. Set
$$p:=m+d(\mu)-d(\pi(\mu))=m+d(\nu)-d(\pi(\nu)).$$
Since
\begin{align*}
d(\sigma^{d(\pi(\mu))}(x))\wedge p&=\big(d(x)-d(\pi(\mu))\big)\wedge \big(m+d(\mu)-d(\pi(\mu))\big)\\
&=\big(d(x)\wedge (m+d(\mu))\big)-d(\pi(\mu))=0,
\end{align*}
\cite[Lemma 2.2]{Rob09} implies $[\sigma^{d(\pi(\mu))}(x);(0,p)]=[z;(0,p)]$. Thus, one may compute
\begin{align*}
\pi(\mu)\circ [z;(0,\infty)]&=[x;(0,d(\pi(\mu)))] \circ [z;(0,p)] \circ [z;(p,\infty)]\\
&=[x;(0,d(\pi(\mu)))] \circ [\sigma^{d(\pi(\mu))}(x);(0,p)] \circ [z;(p,\infty)]\\
&=[x;(0,m+d(\mu))] \circ [z;(p,\infty)]\\
&= [x;(0,m)] \circ [x;(m,m+d(\mu))] \circ [z;(p,\infty)]\\
&=[x;(0,m)] \circ \mu \circ [z;(p,\infty)].
\end{align*}
Analogously, we have
$$\pi(\nu) \circ [z;(0,\infty)]=[x;(0,m)]\circ \nu \circ [z;(p,\infty)],$$
and hence $\pi(\mu) \circ [z;(0,\infty)]=\pi(\nu) \circ [z;(0,\infty)]$ by applying $\mu\sim_{\ul} \nu$. This establishes $\pi(\mu)\sim_{\ul} \pi(\nu)$.

(2) $\Longrightarrow$ (1): Assume that $\pi(\mu)\sim_{\ul} \pi(\nu)$ and $d(\mu)-d(\pi(\mu))=d(\nu)-d(\pi(\nu))$. We first claim
$$[x;(d(\pi(\mu)),m+d(\pi(\mu)))]=[y;(d(\pi(\nu)),m+d(\pi(\nu)))].$$
To this end, we can check conditions (P1)-(P3) for $\big(x;(d(\pi(\mu)),m+d(\pi(\mu)))\big)\approx \big(y;(d(\pi(\nu)),m+d(\pi(\nu)))\big)$. Indeed, the fact $[x;d(\pi(\mu))]=s(\pi(\mu))=s(\pi(\nu))=[y;d(\pi(\nu))]$ gives (P1), whereas (P2) is trivial because $d(\pi(\mu))\wedge d(x)=d(\pi(\mu))$ and $d(\pi(\nu))\wedge d(y)=d(\pi(\nu))$. Also, the relation
$$m+d(\mu)-d(\pi(\mu))=m+d(\nu)-d(\pi(\nu))$$
yields (P3), therefore the claim holds.

Furthermore, since $[x;(d(\pi(\mu)),\infty)]\in s(\pi(\mu))\ul^\infty$, we can write
\begin{align*}
[x;(0,\infty)]&=\pi(\mu) \circ [x;(d(\pi(\mu)),\infty)]\\
&=\pi(\nu) \circ [x;(d(\pi(\mu)),m+d(\mu))] \circ [x;(m+d(\mu),\infty)]\\
&\hspace{8cm} (\mathrm{by~} \pi(\mu)\sim_{\ul} \pi(\nu) )\\
&=[y;(0,d(\pi(\nu)))] \circ [y;(d(\pi(\nu)), m+d(\nu))] \circ [x;(m+d(\mu),\infty)]\\
&\hspace{8.5cm} (\mathrm{by~the~claim})\\
&=[y;(0, m+d(\nu))] \circ [x;(m+d(\mu),\infty)].
\end{align*}
So, we get $[x;(0,m+d(\nu))]=[y;(0,m+d(\nu))]$  by the factorization property. In particular, $\nu=[y;(m,m+d(\nu))]=[x;(m,m+d(\nu))]$.

Now fix an arbitrary $z\in s(\mu)\ul^\infty$. Then $[x;(d(\pi(\mu)),d(\mu))] \circ z \in s(\pi(\mu))\ul^\infty$. Hence
\begin{align*}
[x;(0,m)]\circ \mu \circ z&= [x;(0,m)]\circ [x;(m,m+d(\mu))] \circ z\\
&= [x;(0,d(\pi(\mu)))] \circ [x;(d(\pi(\mu)), m+d(\mu))] \circ z \\
&=\pi(\mu) \circ [x;(d(\pi(\mu)), m+d(\mu))] \circ z\\
&=\pi(\nu) \circ [x;(d(\pi(\mu)), m+d(\mu))] \circ z\\
&\hspace{5cm} (\mathrm{by~} \pi(\mu)\sim_{\ul} \pi(\nu) )\\
&=[y;(0,d(\pi(\nu)))] \circ [y;(d(\pi(\nu)), m+d(\nu))] \circ z\\
&=[y;(0,m)] \circ [y;(m,m+d(\mu))] \circ z\\
&=[x;(0,m)] \circ \nu \circ z.
\end{align*}
We therefore obtain $\mu \circ z=\nu \circ z$, following statement (1).

(2) $\Longrightarrow$ (3): Fix $y\in s(\pi(\mu))\la^{\leq \infty}$. Then $[y;(0,\infty)]\in s(\pi(\mu))\ul^\infty$. Note that $\pi\big([y;(0,\infty)]\big)=y$ because for every $p\leq q\leq d(y)$,
$$\pi\big([y;(0,\infty)]\big)(p,q)=\pi\big([y;(p,q)]\big)=[y;(p,q)]=y(p,q).$$
So, we have
\begin{align*}
\pi(\mu) \circ [y;(0,\infty)]&= \pi(\nu) \circ [y;(0,\infty)] \hspace{1cm} (\mathrm{by~} \pi(\mu)\sim_{\ul} \pi(\nu))\\
\Longrightarrow \hspace{1cm} \pi\big( \pi(\mu) \circ [y;(0,\infty)]\big)&=\pi\big( \pi(\nu) \circ [y;(0,\infty)]\big)\\
\Longrightarrow \hspace{1cm} \pi(\mu) \circ \pi\big([y;(0,\infty)]\big)&=\pi(\nu) \circ \pi\big([y;(0,\infty)]\big) \hspace{1cm} (\mathrm{by~} \pi\circ \pi=\pi)\\
\Longrightarrow \hspace{3.05cm} \pi(\mu)\circ y &= \pi(\nu) \circ y.
\end{align*}
This follows statement (3).

(3) $\Longrightarrow$ (2): Since $s(\pi(\mu))\in \la^0$, by Lemma \ref{lem2.4}, every infinite path in $s(\pi(\mu))\ul^\infty$ is of the form $[z;(0,\infty)]$ with $z(0)=s(\pi(\mu))$. For every such $[z;(0,\infty)]\in s(\pi(\mu))\ul^\infty$, we have
\begin{align*}
\pi(\mu)\circ [z;(0,\infty)]&=[\pi(\mu)z;(0,\infty)]\\
&=[\pi(\nu)z;(0,\infty)] \hspace{1cm} (\mathrm{by~} \pi(\mu)\sim_{\la} \pi(\nu))\\
&=\pi(\nu)\circ [z;(0,\infty)],
\end{align*}
giving $\pi(\mu)\sim_{\ul} \pi(\nu)$.
\end{proof}

The initial definition of maximal tails comes from \cite{bat00} for the 1-graph $C^*$-algebras (see \cite{kan14,car14} also).

\begin{Definition}\label{defn3.6}
Let $\Lambda$ be a row-finite $k$-graph. A \emph{maximal tail} in $\la$ is a nonempty subset \emph{T} of $\Lambda^{0}$ satisfying the following conditions:
\begin{enumerate}
\item for every $\lambda\in\la$, $s(\lambda)\in T$ implies $r(\lambda)\in T$,
\item for every $v\in T$ and $n\in\N^k$, there exists $\lambda\in v\la^{\leq n}$ such that $s(\lambda)\in T$, and
\item for every $v,w\in T$, there exist $\mu\in v\la$ and $\nu\in w\la$ such that $s(\mu)=s(\nu)$.
\end{enumerate}
It is clear that if $T$ is a maximal tail in $\la$, then $\la^0\setminus T$ is hereditary and saturated. We write $\mathrm{MT}(\la)$ for the collection of maximal tails in $\la$.
\end{Definition}

\begin{cor}\label{cor3.7}
Let $\la$ be a locally convex row-finite $k$-graph, and let $\overline{\la}$ be the desourscification of $\la$. Then $\mathrm{Per}(\la)=\mathrm{Per}(\ul)$. Moreover, if $\la^0$ is a maximal tail, then $\mathrm{Per}(\la)$ is a subgroup of $\mathbb{Z}^k$.
\end{cor}

\begin{proof}
For every $\mu,\nu\in \la$, Proposition \ref{prop3.5} implies that $\mu\sim_{\la} \nu$ if and only if $\mu\sim_{\ul} \nu$. This follows $\mathrm{Per}(\la)\subseteq \mathrm{Per}(\ul)$. For the reverse containment, suppose $\mu,\nu\in\ul$, $\mu\sim_{\ul} \nu$ and $d(\mu)-d(\nu) \in \mathrm{Per}(\ul)$. Implication (1) $\Rightarrow$ (3) of Proposition \ref{prop3.5} yields that $\pi(\mu)\sim_{\la} \pi(\nu)$ and $d(\mu)-d(\pi(\mu))=d(\nu)-d(\pi(\nu))$. Thus
$$d(\mu)-d(\nu)=d(\pi(\mu))-d(\pi(\nu))\in \mathrm{Per}(\la),$$
and we have $\mathrm{Per}(\ul)\subseteq \mathrm{Per}(\la)$.

By \cite[Lemma 4.3]{rad17}, if $\la^0$ is a maximal tail, then so is $\ul^0$. Thus, the second statement follows immediately from the first and \cite[Theorem 4.2(1)]{car14}.
\end{proof}

\begin{Definition}[\cite{car14}]
Given a row-finite $k$-graph $\la$, we define the set
\begin{multline*}
H_{\mathrm{Per}(\la)}:=\big\{v\in\la^0: \mathrm{for ~ all}~ \mu \in v\la ~ \mathrm{and} ~ m\in \mathbb{N}^k ~ \mathrm{with ~}\\ d(\mu)-m\in \mathrm{Per}(\la), ~ \mathrm{there ~ exists ~} \nu\in v\la^{m} ~ \mathrm{such ~ that ~} \mu\sim_{\la}\nu \big\}.
\end{multline*}
We will show in Corollary \ref{cor3.10} that $H_{\mathrm{Per}(\la)}$ is a hereditary subset of $\la^0$.
\end{Definition}

\begin{lem}\label{lem3.9}
Let $\la$ be a locally convex row-finite $k$-graph such that $\la^0$ is a maximal tail. Let $\ul$ be the desourcification of $\la$. Then for each $v\in \ul^0$, $v\in H_{\mathrm{Per}(\ul)}$ if and only if $\pi(v)\in H_{\mathrm{Per}(\ul)}$.
\end{lem}

\begin{proof}
We first fix $v\in \ul^0$ with $\pi(v)\in H_{\mathrm{Per}(\ul)}$. Since $\ul$ has no sources  and $\ul^0$ is a maximal tail by \cite[Lemma 4.3]{rad17}, Theorem 4.2 of \cite{car14} says that $H_{\mathrm{Per}(\ul)}$ is a hereditary subset of $\ul^0$. So, we have $v\in H_{\mathrm{Per}(\ul)}$ because $\pi(v)\ul v\neq \emptyset$.

Conversely, assume $v\in H_{\mathrm{Per}(\ul)}$. Let $\mu\in \pi(v)\ul$ and $m\in \N^k$ such that $d(\mu)-m\in \mathrm{Per}(\ul)$. To show $\pi(v)\in H_{\mathrm{Per}(\ul)}$, we must find $\nu \in \pi(v)\ul^m$ such that $\mu\sim_{\ul} \nu$. Write $v=[x;p]$ for some $x\in \la^{\leq \infty}$ and $p\in \N^k$ with $p\wedge d(x)=0$. Let $\mu=[y;(0,d(\mu))]$ for $y\in \pi(v) \la^{\leq \infty}$. Since $x,y\in \pi(v)\la^{\leq \infty}$ and $p\wedge d(x)=0$, \cite[Lemma 2.2]{Rob09} implies $p\wedge d(y)=0$ and $[x;(0,p)]=[y;(0,p)]$. In particular, $\pi(v)=[x;0]=[y;0]$ and $v=[x;p]=[y;p]$. If we set $\mu':=[y;(p,p+d(\mu))]$, then $\mu'\in v\ul^{d(\mu)}$ and $d(\mu')-m=d(\mu)-m\in \mathrm{Per}(\ul)$. So, there exists $\nu'\in v\ul^{m}$ such that $\mu'\sim_{\ul}\nu'$, which is of the form $\nu'=[y;(p,p+m)]$ by Lemma \ref{lem3.4}.

We will show that $\nu:=[y;(0,m)]$ is the desired path (i.e., $\mu\sim_{\ul}\nu$). For, by Proposition \ref{prop3.5}, it suffices to prove $\pi(\mu)\sim_{\ul}\pi(\nu)$ and $d(\mu)-d(\pi(\mu))=d(\nu)-d(\pi(\nu))$. Note that $\pi(\nu)=[y;(0,m\wedge d(y))]=\pi(\nu')$ and
$$d(\nu)-d(\pi(\nu))=m-m\wedge d(y)=d(\nu')-d(\pi(\nu')).$$
Thus
\begin{align*}
d(\mu)-d(\pi(\mu))&=d(\mu)-d(\mu)\wedge d(y)\\
&=d(\mu')-d(\pi(\mu'))\\
&=d(\nu')-d(\pi(\nu')) \hspace{1cm} (\mathrm{by} ~ \mu'\sim_{\ul}\nu' ~\mathrm{and ~ Proposition ~ \ref{prop3.5}})\\
&=d(\nu)-d(\pi(\nu)).
\end{align*}
Furthermore, since $\pi(\mu')=[y;(0,d(\mu)\wedge d(y))]=\pi(\mu)$, implication (1) $\Rightarrow$ (2) of Proposition \ref{prop3.5} yields $\pi(\mu)=\pi(\mu')\sim_{\ul}\pi(\nu')=\pi(\nu)$. Consequently, $\mu\sim_{\ul}\nu$ and we have $\pi(v)\in H_{\mathrm{Per}(\ul)}$.
\end{proof}

\begin{cor}\label{cor3.10}
Let $\la$ be a locally convex row-finite $k$-graph and $\overline{\la}$ its desourcification. Then $H_{\mathrm{Per}(\ul)}\cap \la^0=H_{\mathrm{Per}(\la)}$. In particular, if $\la^0$ is a maximal tail, then $H_{\mathrm{Per}(\la)}$ is a nonempty hereditary subset of $\la^0$.
\end{cor}

\begin{proof}
To see $H_{\mathrm{Per}(\ul)}\cap \la^0\subseteq H_{\mathrm{Per}(\la)}$, take $v\in H_{\mathrm{Per}(\ul)}\cap \la^0$. Suppose $\lambda \in v\la$ and $m\in \N^k$ such that $d(\lambda)-m\in \mathrm{Per}(\la)$. Since $d(\lambda)-m\in \mathrm{Per}(\la)=\mathrm{Per}(\ul)$, there exists $\mu\in v\ul^m$ such that $\lambda\sim_{\ul} \mu$. As $s(\mu)=s(\lambda)\in \la^0$, we have $\mu=\pi(\mu)\in \la$, and hence $\lambda\sim_\la \mu$ by Proposition \ref{prop3.5}. This follows $v\in H_{\mathrm{Per}(\la)}$.

To show $H_{\mathrm{Per}(\la)}\subseteq H_{\mathrm{Per}(\ul)}\cap \la^0$, we fix $v\in H_{\mathrm{Per}(\la)}$. Let $\lambda\in v\ul$ and $m\in \N^k$ such that $d(\lambda)-m\in \mathrm{Per}(\ul)$. Then $\pi(\lambda)\in v\la$, and for $m':=m-\big(d(\lambda)-d(\pi(\lambda))\big)$ we have $d(\pi(\lambda))-m'=d(\lambda)-m\in \mathrm{Per}(\ul)=\mathrm{Per}(\la)$. Since $v\in H_{\mathrm{Per}(\la)}$, there exists $\mu'\in v\la^{m'}$ such that $\pi(\lambda)\sim_\la \mu'$. Let $\lambda=[x;(0,d(\lambda))]$ for some $x\in v \la^{\leq\infty}$. Then for $\mu:=\mu'\circ [x;(d(\pi(\lambda)), d(\lambda))]\in v\ul^{m}$ we have
\begin{align*}
\pi(\mu)&=\pi\big(\mu'\circ [x;(d(\pi(\lambda)), d(\lambda))]\big)\\
&=\pi(\mu')\circ \pi\big([x;(d(\pi(\lambda)), d(\lambda))]\big)\\
&=\mu'\circ [x;d(\pi(\lambda))]\\
&=\mu',
\end{align*}
establishing $\pi(\lambda)\sim_\la \pi(\mu)$ and $d(\lambda)-d(\pi(\lambda))=d(\mu)-d(\pi(\mu))$. Now Proposition \ref{prop3.5} implies $\lambda\sim_{\ul} \mu$, and therefore $v\in H_{\mathrm{Per}(\ul)}$.

The second statement follows from \cite[Theorem 4.2(2)]{car14}. Indeed, if $\la^0$ is a maximal tail, then so is $\overline{\la}^0$ by
\cite[Lemma 4.3]{rad17}. Hence, \cite[Theorem 4.2]{car14} implies that $H_{\mathrm{Per}(\overline{\la})}$ is a nonempty hereditary subset of $\overline{\la}^0$. Therefore, $H_{\mathrm{Per}(\la)}=H_{\mathrm{Per}(\overline{\la})}\cap \la^0$ is a hereditary subset of $\la^0$ as well. Furthermore, we may apply Lemma \ref{lem3.9} to see that $H_{\mathrm{Per}(\la)}=H_{\mathrm{Per}(\overline{\la})}\cap \la^0 \neq\emptyset$.
\end{proof}


\section{primitive ideals in higher-rank graph $C^\ast$-algebras}

In this section, we characterize the primitive ideal space of $C^*(\la)$ for any locally convex, row-finite $k$-graph $\la$. Our results generalize \cite[Theorem 5.3]{car14} and \cite[Corollary 5.4]{car14}.

\begin{Definition}\label{defn4.1}
Let $\la$ be a locally convex row-finite $k$-graph. A boundary path $x\in \la^{\leq \infty}$ is called \emph{cofinal} in case for every $v\in \la^0$, there exists $n\in \mathbb{N}^k$ such that $n\leq d(x)$ and $v\la x(n)\neq \emptyset$. If $x\in \la^{\leq \infty}$ is cofinal, we set
$$F(x):=\{\lambda\sigma^n(x):n\leq d(x) ~ \mathrm{and} ~ \lambda\in \la x(n)\}.$$
Notice that $y\in F(x)$ if and only if $\sigma^m(x)=\sigma^n(y)$ for some $m,n\in \N^k$.
\end{Definition}

\begin{Remark}
In \cite{car14}, the set $F(x)$ was denoted by $[x]$. Since $\ul$ is defined by equivalent classes, we use the notation $F(x)$ instead of $[x]$ here.
\end{Remark}

The proof of \cite[Proposition 3.5]{Rob09} derives the following.

\begin{lem}\label{lem4.3}
Let $x$ be a boundary path in $\Lambda$. Then $x$ is cofinal in $\la$ if and only if the infinite path $[x;(0,\infty)]$ in $\overline{\la}$ is cofinal.
\end{lem}

Recall from Corollary \ref{cor3.7} that if $T$ is a maximal tail in $\la$, then $\mathrm{Per}(\la T)$ is a subgroup of $\mathbb{Z}^k$. Let $\widehat{\mathrm{Per}(\la T)}$ denote \emph{the character space of} $\mathrm{Per}(\la T)$. So, for each character $\eta$ of $\mathrm{Per}(\la T)$, there exists $t\in \mathbb{T}^k \cong \widehat{\Z^k}$ such that $\eta(m)=t^m$ for all $m\in \mathrm{Per}(\la T)$.

\begin{lem}\label{lem4.4}
Let $\la$ be a locally convex, row-finite $k$-graph and  $T$ a maximal tail in $\la$. Suppose that $x$ is a cofinal boundary path in the $k$-subgraph $\la T$ and consider the set $F(x)$ in $\la T$ as in Definition \ref{defn4.1}. Let $\eta\in \per$ and select some $t\in \T^k$ such that $\eta(m)=t^m$ for $m\in \mathrm{Per}(\la T)$. Then the representation $\pi_{x,t}:C^*(\la)\rightarrow \mathcal{B}(\ell^2(F(x)))$ defined by
\begin{equation}\label{equ3.1}
\pi_{x,t}(s_\lambda)\xi_y=\left\{
                              \begin{array}{ll}
                                t^{d(\lambda)}\xi_{\lambda y} & s(\lambda)=y(0) \\
                                0 & otherwise
                              \end{array}
                            \right.
\hspace{1cm}
(\lambda\in \la,~~ y\in F(x))
\end{equation}
is an irreducible representation on $C^*(\la)$.
\end{lem}

\begin{proof}
We will apply the desourcifying method on \cite[Theorem 5.3(1)]{car14}. Let $\overline{\la}$ be the desourcification of $\la$. Recall that $H:=\la^0\setminus T$ is a hereditary and saturated subset of $\la^0$. We first claim that $\overline{H}:=(H\overline{\la})^0$ is hereditary and saturated in $\overline{\la}$. Indeed, the hereditary property of $\overline{H}$ follows from that of $H$ and the fact $\pi(\overline{H})=(\pi(H\ul))^0=(H \pi(\ul))^0=(H\la)^0=H$. To see that $\overline{H}$ is saturated, fix $v\in \overline{\la}^0$ with $s(v\overline{\la}^n)\subseteq \overline{H}$ for some $n\in\mathbb{N}^k$. Then, by $\pi(v\overline{\la}^n)=\pi(v)\la^{\leq n}$, we have
$$s(\pi(v)\la^{\leq n})=s(\pi(v\overline{\la}^n))=\pi(s(v\overline{\la}^n))\subseteq \pi(\overline{H})=H,$$
and thus $\pi(v)\in H\subseteq \overline{H}$ as $H$ is saturated. The fact $\pi(v)\overline{\la} v\neq \emptyset$ combining with the hereditary property of $\overline{H}$ yield $v\in \overline{H}$. Consequently, $\overline{H}$ is saturated and the claim holds.

Write $\overline{T}:=\overline{\la}^0\setminus \overline{H}$. Since $\overline{\la}~\overline{T}=\overline{\la}\setminus \overline{\la}~ \overline{H}$ contains no sources, it is the desourcification of $\la T$; which means $\overline{\la T}=\overline{\la}~\overline{T}$. In particular, $\overline{T}$ is a maximal tail \cite[Lemma 4.3]{rad17}. Moreover, $\overline{x}:=[x;(0,\infty)]$ is a cofinal infinite path in $\overline{\la T}$ by Lemma \ref{lem4.3}. As $\eta\in \widehat{\mathrm{Per}(\la T)}=\widehat{\mathrm{Per}(\overline{\la T})}$, \cite[Theorem 5.3(1)]{car14} ensures that the representation $\pi_{\overline{x},t}:C^*(\overline{\la})\rightarrow \mathcal{B}(\ell^2(F(\overline{x}))$ satisfying
$$\pi_{\overline{x},t}(s_\lambda)\xi_y=\left\{
                              \begin{array}{ll}
          t^{d(\lambda)}\xi_{\lambda y} & s(\lambda)=y(0) \\
                                0 & \mathrm{otherwise}
                              \end{array}
                            \right.
\hspace{1cm} (\lambda\in\overline{\la},~ y\in \overline{\la}^\infty)
$$
is irreducible. Let $\phi:C^*(\la)\rightarrow C^*(\overline{\la})$ be the natural embedding map and define $\pi_{x,t}:=\pi_{\overline{x},t} \circ \phi:C^*(\la)\rightarrow \mathcal{B}(\ell^2(F(\overline{x})))$. Observe that we may regard $\ell^2(F(x))$ as a subspace of $\ell^2(F(\overline{x}))$ by identifying $\xi_y$ with $\xi_{[y;(0,\infty)]}$ for $y\in F(x)$. By the fact $\pi_{x,t}(s_\lambda)\xi_y=0$ unless $s(\lambda)=y(0)$, $\ell^2(F(x))$ is an invariant subspace for operators $\pi(s_\lambda)$, $\lambda\in\la$. Therefore, we can restrict operators $\pi_{x,t}(s_\lambda)$ on $\ell^2(F(x))$ to get the desired representation $\pi_{x,t}:C^*(\la)\rightarrow \mathcal{B}(\ell^2(F(x)))$ satisfying formula (\ref{equ3.1}).
\end{proof}

Now we have all requirements to characterize the primitive ideals of $C^*(\la)$. For any ideal $J$ of $C^*(\la)$, we write $T_J:=\la^0\setminus H_J=\{v\in \la^0: s_v\notin J\}$ in the following.

\begin{thm}\label{thm4.5}
Let $\la$ be a locally convex row-finite $k$-graph.
\begin{enumerate}
  \item If $J$ is a primitive ideal of $C^*(\la)$, then $T:=T_J=\la^0\setminus H_J$ is a maximal tail. Also, there exists a unique $\eta \in \widehat{\mathrm{Per}(\la T)}$ such that the ideal $J\cap (I_{H_J \cup H_{\mathrm{Per}(\la T)}})$ in $C^*(\la)$ is generated by $\big\{s_v:v\in \la^0\setminus T\big\}$ union
      $$\big\{s_\mu- \eta\big(d(\mu)-d(\nu)\big) s_\nu: \mu\sim_{\la T} \nu \mathrm{~and~} r(\mu)=r(\nu)\in H_{\mathrm{Per}(\la T)}\big\}.$$
    Moreover, if $x\in (\la T)^{\leq \infty}$ is cofinal and $t\in \T^k$ with $\eta(m)=t^m$ for $m\in \mathrm{Per}(\la T)$, then we have $J=\ker \pi_{x,t}$.
    \item The map $(T_{\ker \pi_{x,t}}, \eta) \mapsto \ker \pi_{x,t}$ is a bijection between $\bigcup_{T\in \mathrm{MT(\la)}}\big(\{T\}\times \widehat{\mathrm{Per}(\la T)}\big)$ and $\mathrm{Prim} (C^*(\la))$, where $t\in \T^k$ satisfies $\eta(m)=t^m$ for every $m\in \mathrm{Per}(\la T)$.
\end{enumerate}
\end{thm}

\begin{proof}
For statement (1), we fix a primitive ideal $J$ of $C^*(\la)$. Let $\overline{\la}$ be the desourcification of $\la$. Let $\overline{J}$ denote the ideal of $C^*(\ul)$ corresponding with $J$. (Recall that $C^*(\la)$ is a full corner of $C^*(\ul)$ and $\overline{J}=\langle J\rangle$ as an ideal of $C^*(\ul)$.) As the primitivity is preserved under Morita-equivalence, $\overline{J}$ is a primitive ideal of $C^*(\la)$. For simplicity in notation, we set $H:=H_J\subseteq \la^0$, $\overline{H}:=H_{\overline{J}}\subseteq \overline{\la}^0$, $T:=\la^0\setminus H$, and $\overline{T}:=\overline{\la}^0\setminus \overline{H}$. Since $\ul$ is a row-finite $k$-graph with no sources, \cite[Theorem 5.3(2)]{car14} implies that $\overline{T}$ is a maximal tail in $\overline{\la}^0$ and there exists unique $\eta\in \widehat{\mathrm{Per}(\overline{\la} ~\overline{T})}$ such that $\overline{J}=\ker \pi_{\overline{x}, t}$ for any cofinal boundary path $\overline{x}$ in $(H_{\mathrm{Per}(\overline{\la}~\overline{T})}\overline{\la}~\overline{T})^\infty$ and $t\in \T^k$ satisfying $\eta(m)=t^m$ for $m\in \mathrm{Per}(\overline{\la T})$. By Corollary \ref{cor3.10} we have $H_{\mathrm{Per}(\overline{\la}~ \overline{T})}\cap (\la T)^0=H_{\mathrm{Per}(\la T)} \neq \emptyset $, so we may select such $\overline{x}$ with $\overline{x}(0)\in H_{\mathrm{Per}(\la T)}$. As seen in Lemma \ref{lem2.4}, $\overline{x}$ is of the form $[x;(0,\infty)]$ for some $x\in (\la T)^{\leq \infty}$.

Let $\overline{I}_{\overline{H}\cup H_{\mathrm{Per}(\overline{\la} \overline{T})}}$ denote the ideal of $C^*(\ul)$ generated by $\{s_v:v\in \overline{H}\cup H_{\mathrm{Per}(\overline{\la} \overline{T})}\}$. The argument of \cite[Theorem 5.3(2)]{car14} proves that the ideal $\ker \pi_{\overline{x},t} \cap \overline{I}_{\overline{H}\cup H_{\mathrm{Per}(\overline{\la} \overline{T})}}$ in $C^*(\ul)$ is generated by the set
$$B:=\big\{s_v:v\in \overline{H}\big\} \cup \big\{s_\mu- \eta\big(d(\mu)-d(\nu)\big) s_\nu: \mu\sim_{\overline{\la} ~ \overline{T}} \nu , ~ r(\mu)=r(\nu)\in H_{\mathrm{Per}(\overline{\la}~ \overline{T})}\big\}.$$
Let $\psi: C^*(\overline{\la})\rightarrow C^*(\la)$, $a \mapsto P_\la a P_\la$, be the restriction map, where $P_\la:=\sum_{v\in \la^0} s_v\in \mathcal{M}(C^*(\ul))$. Observe that $P_\la$ is a full projection, so $\psi$ induces a one-to-one correspondence between ideal spaces of $C^*(\la)$ and $C^*(\ul)$. In particular, we have
$\psi(\overline{J})=J$ and
$$\psi(\overline{J}\cap \overline{I}_{\overline{H} \cup H_{\mathrm{Per}(\overline{\la}~ \overline{T})}})=J\cap I_{H\cup H_{\mathrm{Per}(\la T)}}.$$
Therefore, the set
$$\psi(B)=\big\{s_v:v\in H\big\} \cup \big\{s_\mu- \eta\big(d(\mu)-d(\nu)\big) s_\nu: \mu\sim_{{\la}T} \nu , r(\mu)=r(\nu)\in H_{\mathrm{Per}({\la} T)}\big\}$$
generates the ideal $J\cap I_{H_{\mathrm{Per}(\la T)}\cup H}$ in $C^*(\la)$.

Furthermore, if $\phi:C^*(\la)\rightarrow C^*(\ul)$ is the embedding map, Lemma \ref{lem4.4} shows that $\pi_{x,t}:=\pi_{\overline{x},t}\circ \phi:C^*(\la)\rightarrow \mathcal{B}(\ell^2 (F(x)))$ is an irreducible representation on $C^*(\la)$. Since
$$\ker \pi_{x,t}=(\pi_{\overline{x},t} \circ \phi)^{-1}(0)=\phi^{-1}(\pi_{\overline{x},t}^{-1}(0))=\phi^{-1}(\overline{J})=J,$$
this completes the proof of statement (1).

Statement (2) follows immediately from (1) combining with \cite[Corollary 5.4]{car14}.
\end{proof}

\begin{note}\label{note4.6}
For every $(T,\eta)\in \bigcup_{T\in \mathrm{MT(\la)}}\big(\{T\}\times \widehat{\mathrm{Per}(\la T)}\big)$, we will denote $J_{(T,\eta)}$ the primitive ideal of $C^*(\la)$ associated to $(T,\eta)$, as described in Theorem \ref{thm4.5}(2).
\end{note}

We say that $\la$ is \emph{strongly aperiodic} if for every saturated hereditary subset $H\subseteq \la^0$, the $k$-subgraph $\la\setminus \la H$ is aperiodic. According to \cite[Theorem 5.3]{rae03}, in case $\la$ is strongly aperiodic, then every ideals of $C^*(\la)$ is gauge-invariant and of the form $I_H$ for some $H\in \mathcal{H}_\la$.

In the following, we write $\mathrm{MT}_a(\la):=\{T\in \mathrm{MT}(\la): \la T ~ \mathrm{is ~ aperidic}\}$.

\begin{cor}\label{cor4.7}
Let $\la$ be a locally convex, row-finite $k$-graph. Then
\begin{enumerate}
  \item The map $T\mapsto I_{\la^0\setminus T}$ is a bijection from $\mathrm{MT}_a(\la)$ onto the collection of primitive gauge-invariant ideals of $C^*(\la)$.
  \item If $\la$ is strongly aperiodic, then $T\mapsto I_{\la^0\setminus T}$ is a bijection between $\mathrm{MT}(\la)$ and $\mathrm{Prim}(C^*(\la))$.
\end{enumerate}
\end{cor}

\begin{proof}
Statement (1): Let $T$ be a maximal tail in $\la$ such that the subgraph $\la T$ is aperiodic. We will show that the gauge-invariant ideal $I_{\la^0\setminus T}$ is primitive. As $\mathrm{Per}(\la T)=\{0\}$ by Proposition \ref{prop3.3}, we have $\widehat{\mathrm{Per}(\la T)}=\{1\}$ and hence $H_{\mathrm{Per}(\la T)}=T$. If $\overline{\la T}$ is the desourcification of $\la T$, \cite[Lemma 5.2]{car14} says that the $k$-graph $\overline{\la T}$ contains a cofinal infinite path $[x;(0,\infty)]$. Moreover, $x$ is a cofinal boundary path in $\la T$ by Lemma \ref{lem4.4}. If we set $J:=\ker \pi_{x,1}$, Theorem \ref{thm4.5}(1) yields $\{v\in \la^0:s_v\in J\}=\la^0 \setminus T$. Write $H:=\la^0\setminus T$ for simplicity. Then ideal $J+I_H$ in the quotient $C^*$-algebra $C^*(\la)/I_H\cong C^*(\la T)$ contains no vertex projections. Since $\la T$ is aperiodic, the Cuntz-Krieger uniqueness theorem implies that $J=I_H$ (cf. \cite[Proposition 4.3]{Rob09}). Therefore $I_{\la^0\setminus T}=\ker \pi_{x,1}$, which is a primitive ideal of $C^*(\la)$ by Lemma \ref{lem4.4}.

Moreover, we may apply the fact $H_{I_H}=H$ for every $H\in \h$ and conclude injectivity of the map.

For surjectivity, we fix a primitive gauge-invariant ideal $I_{\la^0\setminus T}$ (Theorem \ref{thm4.5}(1) says that every primitive gauge-invariant deal of $C^*(\la)$ is of the form $I_{\la^0\setminus T}$ for some maximal tail $T$ in $\la$). Suppose on the contrary $\la T$ is periodic. To derive a contradiction, we will show $I_{\la^0\setminus T}\neq \ker \pi_{x,t}$ for every $x\in (H_{\mathrm{Per}(\la T)}\la T)^{\leq \infty}$ and $t\in \T^k$ described in Theorem \ref{thm4.5}(1). So, let us fix some such $x$ and $t$. Since $\la T$ is periodic, Proposition \ref{prop3.3} implies that there are two distinct paths $\mu,\nu\in \la T$ such that $\mu\sim_{\la T}\nu$. By Theorem \ref{thm4.5}(1), $a=s_\mu- t^{d(\mu)-d(\nu)} s_\nu$ is a nonzero element in $\ker \pi_{x,t}$. But $a\notin I_{\la^0\setminus T}$ because $s(\mu),s(\nu)\in T$  and image of $a$ under the quotient map $C^*(\la)\rightarrow C^*(\la)/I_{\la^0\setminus T}=C^*(\la T)$ is nonzero. Consequently, $\ker \pi_{x,t}\neq I_{\la^0\setminus T}$ as desired.

For statement (2), we notice that in case $\la$ is strongly aperiodic, each quotient $k$-graph $\la T=\la\setminus \la H$ is aperiodic and every ideal of $C^*(\la)$ is gauge-invariant \cite[Theorem 5.3]{rae03}. So, statement (2) is an immediate consequence of (1).
\end{proof}


\section{Maximal ideals}

We know that every ideal of $C^*(\la)$ is the intersection of a family of primitive ones. In particular, maximal ideals of $C^*(\la)$ are all primitive. Here, we want to use the characterization of Theorem \ref{thm4.5} to determine certain maximal ideals of $C^*(\la)$.

\begin{lem}\label{lem5.1}
Let $\Lambda$ be a locally convex row-finite $k$-graph. Suppose that $T$ is a maximal tail in $\la$ and $\eta_1,\eta_2$ are two distinct characters of $\mathrm{Per}(\la T)$. If $J_{(T,\eta_1)}$ and $J_{(T,\eta_2)}$ are respectively the primitive ideals of $C^*(\la)$ corresponding with $(T,\eta_1)$ and $(T,\eta_2)$ (see Notation \ref{note4.6}), then neither $J_{(T,\eta_1)}\subseteq J_{(T,\eta_2)}$ nor $J_{(T,\eta_2)}\subseteq J_{(T,\eta_1)}$.
\end{lem}

\begin{proof}
We suppose $J_{(T,\eta_1)}\subseteq J_{(T,\eta_2)}$ and derive a contradiction; the other case may be discussed analogously. Then, by Theorem \ref{thm4.5}(1), for every $\mu,\nu\in \la T$ with $\mu\sim_{\la T} \nu$, both elements $s_\mu -\eta_1\big(d(\mu)-d(\nu)\big) s_\nu$ and $s_\mu -\eta_2\big(d(\mu)-d(\nu)\big) s_\nu$ belong to $J_{(T,\eta_2)}$. But, Theorem \ref{thm4.5}(1) says that such $\eta_1,\eta_2\in \per$ are unique for $J_{(T,\eta_2)}$, hence $\eta_1=\eta_2$. This contradicts our hypothesis.
\end{proof}

\begin{prop}\label{prop5.2}
Let $\la$ be a locally convex, row-finite $k$-graph. Let $T$ be a maximal tail in $\la$ such that $T$ is minimal in $\mathrm{MT}(\la)$ under $\subseteq$. Then for every $\eta\in \per $, the primitive ideal $J_{(T,\eta)}$ associated to $(T,\eta)$ is a maximal ideal of $C^*(\la)$.
\end{prop}

\begin{proof}
Recall that $H_{J_{(T,\eta)}}=\la^0\setminus T$ for every $\eta\in \per$. Take an ideal $I$ of $C^*(\la)$ such that $J_{(T,\eta)}\subseteq I\subsetneq C^*(\la)$. As $I$ is an intersection of primitive ideals, without loss of generality, we can assume $I$ is primitive. Then by Theorem \ref{thm4.5}(1), $K:=\la^0 \setminus H_I$ is a maximal tail in $\la$ such that $K \subseteq T$. The minimality of $T$ yields either $K=\emptyset$ or $K=T$. If $K=\emptyset$, then $I=C^*(\la)$ which was not assumed. Thus we must have $K=T$. Since $I$ is primitive, Theorem \ref{thm4.5}(1) implies that there exists $\eta'\in \per$ such that $I=J_{(T,\eta')}$. Now apply Lemma \ref{lem5.1} to have $\eta=\eta'$, and hence $I=J_{(T,\eta)}$. Consequently, $J_{(T,\eta)}$ is a maximal ideal of $C^*(\la)$.
\end{proof}

\begin{Definition}
Let $\la$ be a row-finite $k$-graph. We say $\la$ is \emph{cofinal} in case all boundary paths $x\in \la^{\leq \infty}$ are cofinal (in the sense of Definition \ref{defn4.1}). By definition, one may easily verify that $\la$ is cofinal if and only if $\mathcal{H}_\la=\{\emptyset,\la^0\}$.
\end{Definition}

\begin{cor}
Let $\la$ a locally convex row-finite $k$-graph which is cofinal. Then all primitive ideals of $C^*(\la)$, which are of the form $J_{\la^0,\eta}$ for $\eta\in \widehat{\mathrm{Per}(\la)}$, are maximal.
\end{cor}

\begin{proof}
We first claim that $\la^0$ is a maximal tail. To see this, it suffices to check condition (3) of Definition \ref{defn3.6} only. So, fix some $v,w\in \la^0$. If $x\in w\la^{\leq \infty}$, then $x$ is cofinal in $\la$, so there exists $n\leq d(x)$ such that $v\la x(n)\neq \emptyset$. If we select some $\mu\in v\la x(n)$ and set $\nu:=x(0,n)$, we then have $\mu\in v\la$ and $\nu\in w\la$ with $s(\mu)=s(\nu)$. This follows the claim.

Moreover, since $\la$ is cofinal, we have $\mathcal{H}_\la=\{\emptyset, \la^0\}$. Hence $\mathrm{MT}(\la)=\{\la^0\}$, and Theorem \ref{thm4.5} says that primitive ideals of $C^*(\la)$ are of the form $J_{\la^0,\eta}$ for $\eta\in \widehat{\mathrm{Per}(\la)}$. Now Propositions \ref{prop5.2} follows immediately that such ideals are all maximal.
\end{proof}


\section{Decomposability of $C^*(\la)$}

In this section, we investigate the decomposability of a higher-rank graph $C^*$-algebra $C^*(\la)$. Our aim here is to find necessary and sufficient conditions on the underlying $k$-graph $\la$ such that $C^*(\la)$ is decomposable. Furthermore, we show that direct summands in any decomposition of $C^*(\la)$ are themselves isomorphic to higher-rank graph $C^*$-algebras.

\begin{Definition}
We say that $C^*(\la)$ is \emph{decomposable} if there exist two non-zero $C^*$-algebras $A,B$ such that $C^*(\la)=A\oplus B$. Otherwise, $C^*(\la)$ is \emph{indecomposable}. It is clear that in the case $C^*(\la)=A\oplus B$, then $A$ and $B$ are two (closed) ideals of $C^*(\la)$.
\end{Definition}

The key step in our analysis is Corollary \ref{cor6.4} below, which shows that any direct summand in a decomposition of $C^*(\la)$ is a gauge-invariant ideal. To prove this, we use the structure of primitive ideals described in Section 4. Before that, we establish the following two lemmas.

\begin{lem}\label{lem6.2}
Let $T$ be a maximal tail in $\la$. Then the collection $\{J_{(T,\eta)}: \eta\in \per\}$ of primitive ideals of $C^*(\la)$ is invariant under the gauge action $\gamma$.
\end{lem}

\begin{proof}
Fix $\eta\in \per$ and take some $t\in \T^k$ such that $\eta(m)=t^{m}$ for all $m\in \mathrm{Per}(\la T)$. For every $s\in \T^k$, we may define the character $\eta':\mathrm{Per}(\la T) \rightarrow \T$, by $m\mapsto (ts^{-1})^m$. Note that we use the multi-index notation $(ts^{-1})^m:=\prod_{i=1}^k (t_is_i^{-1})^{m_i} \in \T$ here. Thus, for $\mu\sim_{\la T} \nu$ with $r(\mu)=r(\nu)\in H_{\mathrm{Per}(\la T)}$ we have
\begin{align*}
\gamma_s\big(s_\mu-\eta\big(d(\mu)-d(\nu)\big)s_\nu\big)&=s^{d(\mu)}s_\mu-t^{d(\mu)-d(\nu)}\big(s^{d(\nu)}s_\nu\big)\\
&=s^{d(\mu)}\big(s_\mu-(ts^{-1})^{d(\mu)-d(\nu)}s_\nu \big)\\
&=s^{d(\mu)}\big(s_\mu-\eta'\big(d(\mu)-d(\nu)\big)s_\nu \big).
\end{align*}
In view of Theorem \ref{thm4.5}(1), this yields that generators of $\gamma_s\big(J_{(T,\eta)}\big)$ and $J_{(T,\eta')}$ are the same, so $\gamma_s\big(J_{(T,\eta)}\big)=J_{(T,\eta')}$. We are done.
\end{proof}

\begin{lem}\label{lem6.3}
Let $\la$ be locally convex row-finite $k$-graph. If $C^*(\la)$ decomposes as $A\oplus B$, then the collection $D=\{J\in \mathrm{Prim}(C^*(\la)): A\subseteq J\}$ is invariant under the gauge action $\gamma$.
\end{lem}

\begin{proof}
First, note that the decomposability implies that $\mathrm{Prim}(C^*(\la))=\mathrm{Prim}(A)\oplus \mathrm{Prim}(B)$. In particular, $\mathrm{Prim}(A)$ and $\mathrm{Prim}(B)$ are clopen subsets of $\mathrm{Prim}(C^*(\la))$. Moreover, we have $J\in \mathrm{Prim}(B)$ if and only if $A\oplus J\in \mathrm{Prim}(C^*(\la))$, hence $\mathrm{Prim}(B)$ is homeomorphic to $D$.

Let us fix an arbitrary $J_0\in \mathrm{Prim}(C^*(\la))$ with $A\subseteq J_0$ (i.e. $J_0\in D$). By Theorem \ref{thm4.5}, $T:=\{v\in \la^0:s_v\notin J_0\}$ is a maximal tail and there exists $\eta_0\in \per$ such that $J_0=J_{(T,\eta_0)}$. Using Theorem \ref{thm4.5}(2), we may define the homeomorphic embedding $\Psi:\per \rightarrow \mathrm{Prim}(C^*(\la))$, by $\eta \mapsto J_{(T,\eta)}$. Corollary \ref{cor3.7} implies that $\per$ is a subgroup of $\widehat{\Z^k}\cong \T^k$, and thus $\per\cong \T^l$ for some $0\leq l\leq k$. In particular, $\Psi(\per)$ is a connected subset of $\mathrm{Prim}(C^*(\la))$. Since $\Psi(\eta_0)=J_0\in D$ and $D$ is clopen (because $\mathrm{Prim}(B)$ is), this follows that $\Psi(\per)$ must be entirely contained in $D$.

Now, for every $t\in \T^k$, Lemma \ref{lem6.2} yields that $\gamma_t(J_0)=J_{(T,\eta)}$ for some $\eta\in \per$. Since such $J_{(T,\eta)}$ lies in $\mathrm{Rang}(\Psi)\subseteq D$, this concludes the result.
\end{proof}

\begin{cor}\label{cor6.4}
Let $\la$ be a locally convex row-finite $k$-graph. If $C^*(\la)$ decomposes into $C^*(\la)=A\oplus B$, then both $A$ and $B$ are gauge-invariant ideals of $C^*(\la)$.
\end{cor}

\begin{proof}
We know that every ideal of $C^*(\la)$ is the intersection of primitive ideals containing it. So, it suffices to show that the collections $D=\{J\in \mathrm{Prim}(C^*(\la)): A\subseteq J\}$ and $D'=\{J\in \mathrm{Prim}(C^*(\la)): B\subseteq J\}$ are invariant under the gauge action. However, this follows from Lemma \ref{lem6.3} immediately.
\end{proof}

Once we find that $C^*(\la)$ can be decomposed only by gauge-invariant ideals, we may investigate its decomposability by properties of the underlying $k$-graph $\la$.

\begin{Definition}\label{defn6.5}
Let $\la$ be a row-finite $k$-graph. For $v\in \la^0$, denote $T(v):=\{s(\mu):\mu\in v\la\}$ the smallest hereditary subset of $\la^0$ containing $v$. When $H_1,H_2\in \mathcal{H}_\la$ with $H_1\supseteq H_2$, we also set the following subsets of $\la^0$:
\begin{align*}
\Delta(H_1,H_2)&:=\{v\in H_1:T(v)\cap H_2=\emptyset\}, ~ \mathrm{and}\\
\Omega(H_1,H_2)&:=\{v\in H_1\setminus H_2:T(v)\cap H_2\neq \emptyset\}=H_1\setminus \big(H_2\cup \Delta(H_1,H_2)\big).
\end{align*}
Clearly, $\Delta(H_1,H_2)$ is always a hereditary subset of $\la^0$ such that $\Delta(H_1,H_2) \cap H_2= \emptyset$. So, we have $\Sigma(\Delta(H_1,H_2))\cap H_2=\emptyset$ also.

Now we define the relation $\succ$ on $\h$ by: $H_1\succ H_2$ if and only if
\begin{enumerate}
  \item $H_1\supseteq H_2$,
  \item $\Delta(H_1,H_2)\neq \emptyset$ (i.e., there exists $v\in H_1\setminus H_2$ with $T(v)\cap H_2=\emptyset$),  and
  \item for every $v\in \Omega(H_1,H_2)$, there exists $n\in \N^k$ such that $s(v\la^{\leq n})\subseteq H_2 \cup \Delta(H_1,H_2)$.
\end{enumerate}
\end{Definition}

We now determine higher-rank graph $C^*$-algebras $C^*(\la)$ which are decomposable. It is the generalization of \cite[Theorem 4.1]{hon04} for higher-rank graph $C^*$-algebras.

\begin{thm}\label{thm6.6}
Let $\la$ be a locally convex row-finite $k$-graph. Then the following are equivalent.
\begin{enumerate}
  \item $C^*(\la)$ is decomposable.
  \item There exist nonempty, disjoint $H_1,H_2\in \h$ with this property that for every $v\in \la^0\setminus\big( H_1\cup H_2\big)$, there is some $n\in \N^k$ such that $s(v\la^{\leq n})\subseteq H_1\cup H_2$.
  \item There exists nonempty $H\in \h$ such that $\la^0\succ H$.
\end{enumerate}

Moreover, in the case (2) we have
$$C^*(\la)=I_{H_1}\oplus I_{H_2}\cong C^*(\la\setminus \la H_2)\oplus C^*(\la\setminus \la H_1),$$
and in the case (3),
$$C^*(\la)\cong C^*(\la\setminus \la H)\oplus I_H\cong C^*(\la\setminus \la H)\oplus C^*(\la\setminus \la H')$$
where $H'=\Sigma(\Delta(\la^0,H))$.
\end{thm}

\begin{proof}
We will first prove $(1)\Longleftrightarrow(2)$, and then $(2) \Longleftrightarrow (3)$.

(1) $\Longrightarrow$ (2): Let $C^*(\la)=A\oplus B$ be a decomposition for $C^*(\la)$. According to Corollary \ref{cor6.4}, $A$ and $B$ are gauge-invariant ideals of $C^*(\la)$; so, there exist $H_1,H_2\in \h$ such that $A=I_{H_1}$ and $B=I_{H_2}$ \cite[Theorem 5.2(a)]{rae03}. Moreover, we have $H_1\cap H_2=H_{I_{H_1}}\cap H_{I_{H_2}}=\emptyset$, while $\overline{H_1\cup H_2}=\la^0$ because $C^*(\la)=I_{H_1}+ I_{H_2}=I_{\overline{H_1\cup H_2}}$. This follows statement (2).

(2) $\Longrightarrow$ (1): If $H_1,H_2\in \h$ satisfy conditions (2), then $\overline{H_1 \cup H_2}=\la^0$. Thus, \cite[Theorem 5.2(a)]{rae03} implies $I_{H_1}\cap I_{H_2}=I_{H_1\cap H_2}=(0)$ and $I_{H_1}+I_{H_2}=I_{\overline{H_1 \cup H_2}}=C^*(\la)$. Consequently, $C^*(\la)$ decomposes as $I_{H_1}\oplus I_{H_2}$.

(2) $\Longrightarrow$ (3): If $H_1,H_2\in \h$ satisfy (2), it is routine to check that $\la^0\succ H_1,H_2$.

(3) $\Longrightarrow$ (2): Let $\la^0 \succ H$ for some nonempty $H\in \h$. Set $H':=\Sigma(\Delta(\la^0,H))$ the saturated closure of $\Delta(\la^0,H)$. Since $\Delta(\la^0,H)$ is hereditary, $H'$ is a hereditary and saturated subset of $\la^0$ \cite[Lemma 5.1]{rae03}. For each $v\in H'$, there is $n\in \N^k$ such that $s(v\la^{\leq n})\subseteq \Delta(\la^0,H)$, which concludes $H\cap H'=\emptyset$ because $\Delta(\la^0,H)\cap H=\emptyset$. Now the property $\la^0 \succ H$ implies that $H$ and $H'$ satisfy the conditions of (2).

For the last statement, it suffices to note that if $C^*(\la)=I_{H_1}\oplus I_{H_2}$, then
$$I_{H_1}\cong C^*(\la)/I_{H_2}\cong C^*(\la\setminus \la H_2)$$
and analogously $I_{H_2}\cong C^*(\la\setminus \la H_1)$. Now the proof is complete.
\end{proof}


\section{decomposing $C^*(\la)$ with indecomposable components}

In the final section, we want to determine higher-rank graph $C^*$-algebras $C^*(\la)$ which are direct sums of finitely many indecomposable $C^*$-algebras. To this end, we use specific chains of hereditary and saturated subsets of $\la^0$.

\begin{Definition}
Let $\la$ be a row-finite $k$-graph. A sequence $\mathcal{C}:=\{H_i\}_{i=1}^n$, for $n\in \N\cup \{\infty\}$, of saturated hereditary subsets of $\la^0$ is called a \emph{chain in $\h$} whenever $H_1=\la^0$ and $H_i\succ H_{i+1}$ for all $1\leq i<n$ (see Definition \ref{defn6.5}). Then, $n=|\mathcal{C}|$ is \emph{length} of $\mathcal{C}$. A \emph{refinement of $\C$} is a chain $\C'$ in $\h$ such that $\C \subsetneq \C'$. In case there are no refinements for $\C$, we say $\C$ is a \emph{maximal chain}.
\end{Definition}

\begin{lem}\label{lem7.2}
Let $\la$ be a locally convex row-finite $k$-graph. If $\C:\la^0=H_1\succ \cdots \succ H_n$ is a finite chain in $\h$, then there exists a decomposition $C^*(\la)=I_{K_1}\oplus \cdots \oplus I_{K_n}$ such that $K_i$'s are pairwise disjoint and $H_i=\Sigma(\bigcup_{j=i}^n K_j)$ for each $i$. Moreover, $\C$ is a maximal chain if and only if all $I_{K_j}$ are indecomposable.
\end{lem}

\begin{proof}
Since $H_1 \succ H_2$, implication $(3) \Rightarrow  (1)$ of Theorem \ref{thm6.6} implies that $C^*(\la)=I_{K_1} \oplus I_{H_2}$ where $K_1=\Sigma(\Delta(H_1,H_2))$. Then $I_{H_2}\cong C^*(\la\setminus \la K_1)$, and the saturation of $H_2$ in the subgraph $\la\setminus \la K_1$ is all $(\la\setminus \la K_1)^0$. So $(\la\setminus \la K_1)^0 \succ H_3$ in $\mathcal{H}_{\la\setminus \la K_1}$ because $H_2 \succ H_3$, and we have again $I_{H_2} \cong C^*(\la \setminus \la K_1)=I_{K_2} \oplus I_{H_3}$ for some $K_2 \in \mathcal{H}_{\la \setminus \la K_1}$. Since $I_{K_2}$ is a direct summand of $C^*(\la)$, it is a gauge-invariant ideal of $C^*(\la)$ by Corollary \ref{cor6.4}. Hence, $K_2\in \h$. Continuing this process gives a decomposition
$$C^*(\la)=I_{K_1}\oplus I_{K_2} \oplus \cdots \oplus I_{K_n}$$
for $C^*(\la)$, where $K_n=H_n$. The above process says that $I_{H_i}=\bigoplus_{j=i}^n I_{K_j}$ for each $1\leq i\leq n$, which follows $H_i=\Sigma(\bigcup_{j=i}^n K_j)$ by Theorem 5.2(a) of \cite{rae03}.

For the last statement, if some $I_{K_j}$ decomposes as $A\oplus B$, then $A$ and $B$ are two direct summands of $C^*(\la)$; so $I_{K_j}=I_{K_j'}\oplus I_{K_j''}$ for some $\emptyset \neq K_j',K_j''\in \h$ by applying Corollary \ref{cor6.4}. Therefore, we have a refinement
$$\la^0=H_1 \succ \cdots \succ H_j \succ \Sigma\big(H_{j+1} \cup K_j'\big) \succ H_{j+1} \succ \cdots \succ H_n$$
of $\C$. Conversely, if
$$\C': \la^0=H_1 \succ \cdots \succ H_j \succ H \succ H_{j+1} \succ \cdots \succ H_n$$
is a refinement of $\C$, then $I_{K_j}$ will be decomposable by (3) $\Rightarrow$ (1) in Theorem \ref{thm6.6}. This completes the proof.
\end{proof}

\begin{Definition}
Let $\la$ be a locally convex row-finite $k$-graph. We say that $C^*(\la)$ is \emph{n-decomposable}, for $n\geq 1$, if there exists an $n$-term decomposition $C^*(\la)=A_1\oplus \cdots \oplus A_n$ for $C^*(\la)$ such that each direct summand $A_i$ is indecomposable (1-decomposable equals to indecomposable). Note that such $n$, if exists, is unique (see the last paragraph in Proof of Theorem \ref{thm7.4} below).
\end{Definition}

We now characterize $n$-decomposable higher-rank graph $C^*$-algebras.

\begin{thm}\label{thm7.4}
Let $\la$ be a locally convex row-finite $k$-graph. Then $C^*(\la)$ is $n$-decomposable for some $n\geq 1$ if and only if $n=\mathrm{max}\{|\C|:\C \mathrm{~is ~ a~chain~in~}\h\}$. Moreover, if this is the case, then there is a unique decomposition (up to permutation) $C^*(\la)\cong C^*(\la_1)\oplus \cdots \oplus C^*(\la_n)$, where $\la_i$ are $k$-subgraphs of $\la$ and each $C^*(\la_i)$ is indecomposable.
\end{thm}

\begin{Remark}
In view of Lemma \ref{lem6.2}, every maximal chain of length $n\geq 1$ in $\h$ induces an $n$-term decomposition for $C^*(\la)$ with indecomposable direct summands. Since such decompositions are unique up to permutation, Theorem \ref{thm7.4} follows that maximal chains in $\h$ are all of length $n$. In particular, in this case, any chain in $\h$ has a refinement with length $n$.
\end{Remark}

\begin{proof}[Proof of Theorem \ref{thm7.4}.]
($\Longrightarrow$): Let $C^*(\la)=A_1\oplus \cdots \oplus A_n$ be a decomposition of $C^*(\la)$ with indecomposable summands. By Corollary \ref{cor6.4}, each $A_j$ is a gauge-invariant ideal of $C^*(\la)$, so $A_j=I_{K_j}$ for some $K_j\in \h$. If we set $H_i:=\Sigma(\bigcup_{j=i}^n K_j)$ for $1\leq i\leq n$, Lemma \ref{lem7.2} implies that
$$\C:\la^0=H_1\succ \cdots \succ H_n$$
is a chain in $\h$. We hence have $n\leq \max \{|\C|: \C\mathrm{~is~a~chain~in~} \h\}$.

On the other hand, assume $\la^0=H_1\succ \cdots \succ H_l$ is a chain in $\h$. By Lemma \ref{lem7.2}, there is a decomposition $C^*(\la)=I_{K_1}\oplus \cdots \oplus I_{K_l}$ such that $H_i:=\Sigma(\bigcup_{j=i}^l K_j)$ for $1\leq i\leq l$. So, for each $i$,  we have $I_{K_i}=\bigoplus_{j=1}^nI_{K_i}\cap A_j$, and thus either $I_{K_i}\cap A_j=(0)$ or $A_j\subseteq I_{K_i}$ because $A_j$ is indecomposable. This follows $l\leq n$, and consequently $n=\mathrm{max}\{|\C|:\C \mathrm{~is ~ a~chain~in~}\h\}$.

The ``$\Longleftarrow$" part follows from Lemma \ref{lem7.2}. Indeed, if $n=\max\{|\C|: \C \mathrm{~is~a~chain~in~}\h\}$ and $\C$ is a maximal chain in $\h$ with $|\C|=n$, then $\C$ induces a decomposition $C^*(\la)=A_1\oplus \cdots \oplus A_n$ with indecomposable summands. Therefore, $C^*(\la)$ is $n$-decomposable.

For the last statement, suppose that $C^*(\la)$ is $n$-decomposable and $C^*(\la)=I_{K_1}\oplus \cdots \oplus I_{K_n}$. Then, for every $1\leq i\leq n$, we have $C^*(\la)=I_{K_i}\oplus I_{K_i'}$ where $K_i':=\Sigma(\bigcup_{j\neq i} K_j)$; in particular,
$$I_{K_i}\cong \frac{C^*(\la)}{I_{K_i'}}\cong C^*(\la\setminus \la K_i').$$
Therefore, $\la_i:=\la\setminus \la K_i'$ are $k$-subgraphs of $\la$ and we have $C^*(\la)=\bigoplus_{i=1}^n I_{K_i}\cong \bigoplus_{i=1}^n C^*(\la_i)$.

To see that this decomposition is unique, assume $C^*(\la)=A_1\oplus \cdots \oplus A_m$ with indecomposable components. For any $1\leq j\leq m$, $A_j=A_j \cap C^*(\la)=\bigoplus_{i=1}^n A_j\cap I_{K_i}$. Since $A_j$ is indecomposable, there is $1\leq l_j\leq n$ such that $A_j\subseteq I_{K_{l_j}}$ and $A_j\cap I_{K_i}=(0)$ for $i\neq l_j$. By a same argument we have $I_{K_{l_j}}\subseteq A_j$ also, and hence they are equal. This follows that $m=n$ and $\{A_j\}_{j=1}^n$ is a permutation of $\{I_{K_i}\}_{i=1}^n$.
\end{proof}

The following is an immediate consequence of Theorem \ref{thm7.4}.

\begin{cor}
Let $\la$ be a locally convex row-finite $k$-graph. If $\la$ contains only finitely many saturated hereditary subsets, then there exist finitely many $k$-subgraphs $\la_1,\ldots, \la_n$ of $\la$ such that $C^*(\la)\cong C^*(\la_1)\oplus \cdots \oplus C^*(\la_n)$ and each $C^*(\la_i)$ is indecomposable.
\end{cor}

{\bf Acknowledgement.}
The author would like to kindly acknowledge the referee for his/her useful comments and pointing the related article \cite{web11}.



\end{document}